\documentclass[a4paper,11pt]{article}

\usepackage[margin=2cm]{geometry}
\usepackage[utf8]{inputenc}
\usepackage[T1]{fontenc}
\usepackage{amsfonts,amsmath,amssymb,amscd,amstext,amsthm}
\usepackage{textcomp,multicol}
\usepackage[shortlabels]{enumitem}
\usepackage{mathrsfs}\let\mathcal\mathscr
\usepackage{stmaryrd}
\usepackage{xcolor}
\usepackage{tikz}
\usetikzlibrary{decorations.pathreplacing}
\usepackage{hyperref}

\newcommand{\N}{\mathbb{N}}
\let\NM=\N
\renewcommand{\P}{\mathbb{P}}
\let\PM=\P
\newcommand{\Q}{\mathbb{Q}}
\let\QM=\Q

\newcommand{\Z}{\mathbb{Z}}
\let\ZM=\Z

\newcommand{\fg}{\mathbf{f}}

\newcommand{\sg}{\mathbf{s}}
\newcommand{\tg}{\mathbf{t}}
\newcommand{\ug}{\mathbf{u}}
\let\llb=\llbracket
\let\rrb=\rrbracket
\DeclareMathOperator{\GL}{GL}

\DeclareMathOperator{\SL}{SL}
\DeclareMathOperator{\PGL}{PGL}
\DeclareMathOperator{\PSL}{PSL}

\newcommand{\If}{\mathrm{I}}

\let\a=\alpha
\let\ga=\gamma
\let\Ga=\Gamma
\let\eps=\varepsilon

\newtheorem{theo}{Theorem}[section]
\newtheorem{coro}[theo]{Corollary}
\newtheorem{prop}[theo]{Proposition}
\newtheorem{lemm}[theo]{Lemma}

\theoremstyle{definition}
\newtheorem*{defi}{Definition}
\newtheorem{exo}{Exercise}
\newcommand{\be}{\begin{exo}}
\newcommand{\ee}{\end{exo}}

\theoremstyle{remark}
\newtheorem*{nota}{Notation}

\newtheorem*{exas}{Examples}

\newtheorem{exos}{Exercise 1-\!\!}
\newcommand{\bes}{\begin{exos}}
\newcommand{\ees}{\end{exos}}

\theoremstyle{remark}
\newtheorem*{rema}{Remark}

\definecolor{slateblue}{RGB}{72,61,139}

\newcommand{\cdotq}{\underset{q}{\cdot}}

\title{Lattice paths on snake graphs for $q$-deformed rationals}
\author{Jérôme Germoni}

\begin{document}
\maketitle

\paragraph{Abstract.} We give an elementary, straightforward proof of Ovenhouse's interpretation of the denominator and the numerator of a (right) $q$-deformed rational~$[m/n]_q=[m/n]_q^\sharp$ in terms of weighted lattice paths on a ``snake graph'', and we give an analogous result for the left $q$-deformed rational~$[m/n]_q^\flat$ introduced by Bapat, Becker, and Licata. The key ingredient is a simple formula for the $q$-deformation of the matrix of the composition of two homographies of the form $z\mapsto a+1/z$.

\section{Introduction}

Deformations of rationals, the so-called $q$-rationals, were introduced by Sophie Morier-Genoud and Valentin Ovsienko (\cite{MG-O_0}, \cite{MG-O_1}, \cite{MG-O_2}) as a generalisation of Gauss's polynomial deformation of integers (\cite{Gauss}). The construction is based on a deformation of the usual action of~$\PSL_2(\ZM)$ on~$\QM\cup\{\infty\}$ into an action on~$\QM(q)\cup\{\infty\}$. The numerators and the denominators of the fractions thus produced were given a combinatorial interpretation in terms of ideals of a poset associated to a \emph{fence graph}. Later, Nicholas Ovenhouse \cite{Ovenhouse} translated this in terms of weighted paths on a graph called a \emph{snake graph}. Recently, Christophe Aval and Sébastien Labbé~\cite{AL} gave other combinatorial interpretations.

As noticed by Ovsienko~\cite{Ovsienko}, the term \emph{snake graph} was coined by  James Propp~\cite{Propp}, but the graphs themselves have been known for a long time under various names. İlke Çanakçı and Ralf Schiffler used them in~\cite{CS_1}, then they related paths on snake graphs to continued fractions in~\cite{CS_2}. Besides, more recently, Asilata Bapat, Louis Becker, and Anthony~M. Licata introduced a ``left'' deformation of rationals~\cite{BBL}.

In this note, we provide a simple, direct relation between numerators and denominators of~$q$-rationals and weighted paths on snake graphs, i.e. a direct proof of Ovenhouse's results. The key remark is that while the basic operation for continued fractions, namely $z\mapsto a+1/z$, corresponds to a matrix of determinant~$-1$, the composition of two of these can be $q$-deformed (proposition~\ref{A1A2q}), and the corresponding matrix has a simple combinatorial significance (proposition~\ref{A1A2q_chemins}). This procedure allows to skip posets, but also the triangular graphs and the Farey graph that appeared in the initial articles.

In all the article, we denote by $q$ an indeterminate/formal variable. The set~$\NM$ contains zero.

\section{Construction of $q$-deformed rationals}

It is well-known that the group~$\SL_2(\ZM)$ of $2\times2$ integer matrices with determinant~$1$ is generated by
\[S=\begin{pmatrix}0&-1\\1&0\end{pmatrix}\quad\text{and}\quad
T=\begin{pmatrix}1&1\\0&1\end{pmatrix}.\]

These matrices satisfy the relations $S^2=(ST)^3=-\If_2$, where~$\If_2$
is the identity matrix, so that their respective images~$s$ and~$t$
generate the quotient group $\PSL_2(\Z)=\SL_2(\ZM)/\{\pm\If_2\}$, which
is presented by~$s$ and~$t$ and relations
\[s^2=(st)^3=1.\]

There is a very classical action of~$\SL_2(\ZM)$ on~$\P^1(\Q)=\QM\cup\{\infty\}$, where a matrix $A=\left(\begin{smallmatrix}a&b\\c&d\end{smallmatrix}\right)$ acts on~$z$ by $A\cdot z=(az+b)/(cz+d)$. The same notation will be used throughout with~$A$ and~$z$ over other base rings. When we deal with homographies, we follow the usual conventions about infinity, namely: $a/0=\infty$ if~$a\ne0$ and $(a\infty+b)/(c\infty+d)=a/c$ for $a$,~$b$,~$c$,~$d$ in~$\Q$ or~$\QM(q)$ such that $ad-bc\ne0$. The kernel of the $\SL_2(\ZM)$-action is~$\{\pm\If_2\}$, so~$\PSL_2(\Z)$ acts faithfully on~$\PM^1(\QM)$.

The following lemma is trivial, but it allows to stress the role
of~$S_q$ and~$T_q$.

\begin{lemm}
  Let 
\[S_q=\begin{pmatrix}0&-1\\q&0\end{pmatrix}\quad\text{and}\quad
T_q=\begin{pmatrix}q&1\\0&1\end{pmatrix}.\]

We have $S_q^2=-q\If_2$ and $(S_qT_q)^3=-q^3\If_2$.
\end{lemm}

Since $\det S_q=\det T_q=q$, the matrices $S_q$ and $T_q$ live in the group $\GL_2(\ZM[q,q^{-1}])$ of invertible~$2\times2$ matrices over the ring of Laurent polynomials. The center of this group is
\[Z=\bigl\{\eps q^k\If_2\,:\ \eps\in\{-1,1\},\ k\in\ZM\bigr\},\]
and we denote by $\PGL_2(\ZM[q,q^{-1}])$ the quotient group~$\GL_2(\ZM[q,q^{-1}])/Z$, which can be identified with the group of homographies of~$\QM(q)\cup\{\infty\}$. 

\begin{coro}\label{q-action}
  \begin{enumerate}[{\em(i)}]
  \item \label{q-action-i}There is a unique morphism $\SL_2(\ZM)\to\PGL_2(\ZM[q,q^{-1}])$, $A\mapsto [A]_q$, such that $[S]_q=S_q$ and $[T]_q=T_q$. This gives rise to a well-defined action of~$\SL_2(\ZM)$ on $\QM(q)\cup\{\infty\}$, denoted $A\cdotq F$ for~$A\in\SL_2(\ZM)$ and~$F\in\QM(q)\cup\{\infty\}$, such that
    \[S\cdotq F=[S]_q\cdot F=-\frac{1}{qF}\quad\text{and}\quad T\cdotq F=[T]_q\cdot F=qF+1.\]
  \item \label{q-action-ii}The stabilizer of~$\infty$ in $\SL_2(\ZM)$ is the subgroup generated by~$T$ and~$S^2=-\If_2$. With respect to the action of assertion~\ref{q-action-i}, this subgroup fixes both~$\infty$ and~$1/(1-q)$.
  \end{enumerate}
\end{coro}

Thanks to the corollary, the following definition makes sense.

\begin{defi}
  Let $m$, $n$ be relatively prime integers (so~$n\ne0$). Let~$(u,v)\in\ZM^2$ such that $um+vn=1$. The left and right $q$-deformations of the rational~$m/n$ are respectively defined by
  \[\left[\frac{m}{n}\right]_q^\flat=\begin{pmatrix}m&-v\\n&u\end{pmatrix}\cdotq\frac{1}{1-q}\quad\text{and}\quad \left[\frac{m}{n}\right]_q^\sharp=\begin{pmatrix}m&-v\\n&u\end{pmatrix}\cdotq\infty.\]
  
  This makes sense thanks to assertion~{\em\ref{q-action-ii}} of proposition~\ref{q-action}, since the matrix~$\left(\begin{smallmatrix}m&-v\\n&u\end{smallmatrix}\right)$ is unique up to a sign ($S^2=-\If_2$) and right multiplication by a power of~$T$, and the matrices~$S^2$ and~$T$ fix both~$1/(1-q)$ and~$\infty$.
\end{defi}

\begin{exas}
  For $n=1$, one can take $v=1$ and $u=0$. Since $\left(\begin{smallmatrix}m&-1\\1&0\end{smallmatrix}\right)=T^mS$, an easy induction shows that the $q$-deformed integers are respectively
  \[[m]_q^\flat=q^m+q^{m-2}+q^{m-3}\cdots+q+1\quad\text{and}\quad [m]_q^\sharp=q^{m-1}+q^{m-2}+\cdots+q+1.\]
  Fortunately, this is coherent with \cite[example 1.2 (a)]{MG-O_1} and~\cite[\S~2.3]{BBL}. For the sake of simplicity, we will abbreviate $[m]^\sharp_q=[m]_q$.
\end{exas}

\section{Integer matrix computations}

For $a,c\in\ZM$, let~$A_a$ and~$C_c$ be the following matrices:
\[  A_a=\begin{pmatrix}a&1\\1&0\end{pmatrix}\quad\text{and}\quad
    C_c=\begin{pmatrix}c&-1\\1&0\end{pmatrix}.\]

These matrices encode the key operation in continued fractions, both
classical and \emph{à la Hirzebruch} (following~\cite[p.~6]{MG-O_2}): for~$z\in\QM\cup\{\infty\}$,
\[
  a+\frac1z=\frac{az+1}{1z+0}=A_a\cdot z\quad\text{and}\quad
  c-\frac1z=\frac{cz-1}{1z+0}=C_c\cdot z.\]

\begin{nota}
  Let $(a_1,\dots,a_k)$ (with $a_i\ge1$ for $i\ge2$) and $(c_1,\dots,c_l)$ (with $c_j\ge2$ for $j\ge2$) be two finite sequences of integers. One denotes  the following continued fraction expansions:
  \begin{align*}
    [a_1,\dots,a_k]&=A_{a_1}A_{a_2}\cdots A_{a_k}\cdot\infty=a_1+\dfrac{1}{a_2+\dfrac{1}{\ddots+\dfrac{1}{a_{k}}}};\\
    \llb c_1,\dots,c_l\rrb&=C_{c_1}C_{c_2}\cdots C_{c_l}\cdot\infty=c_1-\dfrac{1}{c_2-\dfrac{1}{\ddots-\dfrac{1}{c_{l}}}}.
  \end{align*}
  
  Accordingly, when $k=2p$ is even, we will denote by $[a_1,\dots,a_{2p}]_q^\flat$ and $[a_1,\dots,a_{2p}]_q^\sharp$ the~$q$-deformations of the rational number~$[a_1,\dots,a_{2p}]$ (this makes sense, since $\det(A_{a_1}\cdots A_{a_k})=1$).
\end{nota}

The next lemma follows from a straightforward induction.

\begin{lemm}
  One has
  \begin{align*}
    T^a
    &=\begin{pmatrix}1&a\\0&1\end{pmatrix}&&(a\in\Z);\\
    C_2
    &=T^2S=\begin{pmatrix}2&-1\\1&0\end{pmatrix};\\
    C_2^c
    &=\begin{pmatrix}c+1&-c\\c&-c+1\end{pmatrix}&&(c\in\N,\ c\ge1).
  \end{align*}
\end{lemm}

\begin{rema}
  The matrices~$A_a$ have determinant~$-1$, so they are not subject to corollary~\ref{q-action}, but the product~$A_{a_1}A_{a_2}$ of two such matrices \emph{is} in~$\SL_2(\Z)$. To guess an expression in terms of~$S$ and~$T$, one computes
  \[A_{a_1}A_{a_2}=\begin{pmatrix}a_1&1\\1&0\end{pmatrix}\begin{pmatrix}a_2&1\\1&0\end{pmatrix}
  =\begin{pmatrix}a_1a_2+1&a_1\\a_2&1\end{pmatrix},\]
  which leads to find a continued fraction à la Hirzebruch
  for~$\frac{a_1a_2+1}{a_2}$:
  \[\frac{a_1a_2+1}{a_2}=a_1+\frac1{a_2}=\underbrace{a_1+1}_c-\frac{1}{\;\frac{a_2}{a_2-1}\;}
  =a_1+1-\frac{1}{2-\frac{a_2-1}{a_2-2}}
  =a_1+1-\frac{1}{2-\frac{1}{2-\frac{a_2-2}{a_2-3}}},\quad\text{etc.}\]

  This computation gives~$A_{a_1}A_{a_2}$ up to an element of the stabilizer of~$\infty$ (more or less a power of~$T$). It turns out that the element is~$T$.
\end{rema}

\begin{lemm}\label{A1A2}
  For nonnegative integers $a_1$, $a_2$, one has
  \[A_{a_1}A_{a_2}=T^{a_1+1}S(T^2S)^{a_2-1}T.\]
\end{lemm}

\begin{proof}
  This is a straightforward verification:
\[T^{a_1+1}S(T^2S)^{a_2-1}T=C_{a_1+1}C_2^{a_2-1}T
    =\begin{pmatrix}a_1a_2+a_2-a_2+1&a_1+1-1\\a_2&1\end{pmatrix}
    =A_{a_1}A_{a_2}.\qedhere
    \]
\end{proof}

\section{Deformed matrix computations}

\begin{lemm}\label{powers_TC}
  \begin{enumerate}[{\em(i)}]
  \item Let $a\in\NM$. One has
    \[T_q^a=\begin{pmatrix}q^a&[a]_q\\0&1\end{pmatrix}.\]
  \item One has:
    \[ T_q^2S_q=q\begin{pmatrix}q+1&-q\\1&0\end{pmatrix}\]
    and, more generally, for $c\in\N$, $c\ge1$:
    \[(T_q^2S_q)^c=q^c\begin{pmatrix}[c+1]_q&-q[c]_q\\ [ c ]_q&-q[c-1]_q\end{pmatrix}.\]
  \end{enumerate}
\end{lemm}
\begin{proof}
  (i) There is nothing to show for $a=1$ (recall that $[1]_q=1$). If the
  relation is verified for some integer~$a$, then
  \[T_q^{a+1}=\begin{pmatrix}q&1\\0&1\end{pmatrix}\begin{pmatrix}q^{a+1}&q[a]_q+1\\0&1\end{pmatrix}=\begin{pmatrix}q^{a+1}&[a+1]_q\\0&1\end{pmatrix}.\]

  (ii) First of all, thanks to the previous computation of~$T_q^2$, one has:
  \[T_q^2S_q=\begin{pmatrix}q^2&q+1\\0&1\end{pmatrix}\begin{pmatrix}0&-1\\q&0\end{pmatrix}
    =\begin{pmatrix}q(q+1)&-q^2\\q&0\end{pmatrix},\]
  which is the case $c=1$. If the relation is verified for some integer~$c$, then
  \begin{align*}(T_q^2S_q)^c
    &=q\begin{pmatrix}q+1&-q\\1&0\end{pmatrix}q^c\begin{pmatrix}[c+1]_q&-q[c]_q\\ [ c ]_q&-q[c-1]_q\end{pmatrix}\\
    &=q^{c+1}\begin{pmatrix}q[c+1]_q+[c+1]_q-q[c]_q&-q^2[c]_q-q[c]_q+q^2[c-1]_q\\ [c+1]_q&-q[c]_q\end{pmatrix},
  \end{align*}
  and one can conclude by using $[c+1]_q-q[c]_q=1$ and
  $q[c+1]_q+1=[c+2]_q$ for the first coefficient, and, similarly,
  $-q[c]_q+q^2[c-1]_q=-q\bigl([c]_q-q[c-1]_q\bigr)=-q$ and
  $-q^2[c]_q-q=-q\bigl(q[c]_q+1\bigr)=-q[c+1]_q$ for the second one.
\end{proof}

\begin{prop}\label{A1A2q}
  Let~$a_1$ and~$a_2$ be positive integers. Up to a sign and a power of~$q$, one has
  \[[A_1A_2]_q=T_q^{a_1+1}S_q(T_q^2S_q)^{a_2-1}T_q=
    \begin{pmatrix}q[a_1][a_2]+q^{a_1+a_2}&[a_1]\\q[a_2]&1\end{pmatrix}.
  \]  
\end{prop}

\begin{proof}
  The first equality (up to~$\pm q^k$) comes from lemma~\ref{A1A2}, given that the map $A\mapsto[A]_q$ is the morphism such that $[S]_q=S_q$ and $[T]_q=T_q$.
  The second equality is a somewhat tedious, but straightforward verification: up to~$\pm q^k$,
  \begin{align*}B
    &:=T_q^{a_1+1}S_q(T_q^2S_q)^{a_2-1}T_q\\
    &\hphantom{:}=\begin{pmatrix}q^{a_1+1}&[a_1+1]\\0&1\end{pmatrix}\begin{pmatrix}0&-1\\q&0\end{pmatrix}
      \begin{pmatrix}[a_2]&-q[a_2-1]\\ [a_2-1]&-q[a_2-2]\end{pmatrix}\begin{pmatrix}q&1\\0&1\end{pmatrix}\\
    &\hphantom{:}=
      \begin{pmatrix}q^{a_1+1}&[a_1+1]\\0&1\end{pmatrix}\begin{pmatrix}0&-1\\q&0\end{pmatrix}\begin{pmatrix}q[a_2]&1\\q[a_2-1]&1\end{pmatrix}\\
    &\hphantom{:}=
      \begin{pmatrix}q^{a_1+1}&[a_1+1]\\0&1\end{pmatrix}\begin{pmatrix}-q[a_2-1]&-1\\q^2[a_2]&q\end{pmatrix}\\
    &\hphantom{:}=
      \begin{pmatrix}-q^{a_1+2}[a_2-1]+q^2[a_1+1][a_2]&-q^{a_1+1}+q[a_1+1]\\q^2[a_2]&q\end{pmatrix}\\
    &=\hphantom{:}
      \begin{pmatrix}-q^{a_1+1}\bigl([a_2]-q^{a_2-1}\bigr)+q\bigl([a_1]+q^{a_1}\bigr)[a_2]&[a_1+1]-q^{a_1+1}\\q[a_2]&1\end{pmatrix}\\
    &=\hphantom{:}
      \begin{pmatrix}q[a_1][a_2]+q^{a_1+a_2}&[a_1]\\q[a_2]&1\end{pmatrix}.\qedhere
  \end{align*}
\end{proof}

One can recover the expression of the right $q$-deformed rational~$[m/n]_q^\sharp$ given in~\cite[exercise~4]{MG-O_2}.

\begin{coro}
  Let $(a_1,\dots,a_{2p})$ be an even length tuple of positive
  integers. The right $q$-deformed rational associated
  to~$\frac{m}{n}=[a_1,\dots,a_{2p}]$ is
  \[ \left[\frac{m}{n}\right]_q^\sharp=[a_1]_q+\frac{q^{a_1}}{[a_2]_{q^{-1}}+\frac{q^{-a_2}}{[a_3]_q+\frac{q^{a_3}}{[a_4]_{q^{-1}}+\frac{q^{-a_4}}{\frac{\ddots}{[a_{2p-1}]_q+\frac{q^{a_{2p-1}}}{[a_{2p}]_{q^{-1}}}}}}}}.\]
\end{coro}

\begin{proof}
  First assume $p=1$. Since
  $a_1+\frac{1}{a_2}=\frac{a_1a_2+1}{a_2}=A_{a_1}A_{a_2}\cdot\infty$,
  one lets the matrix $[A_{a_1}A_{a_2}]_q$ computed in
  proposition~\ref{A1A2q} act on some~$z\in\Q(q)\cup\{\infty\}$:
  \[
    [A_{a_1}A_{a_2}]_q\cdot z=\frac{\bigl(q[a_1]_q[a_2]_q+q^{a_1+a_2}\bigr)z+[a_1]_q}{q[a_2]_qz+1}
    =[a_1]_q+\frac{q^{a_1}}{q^{-a_2+1}[a_2]_q+\dfrac{q^{-a_2}}z}
    =[a_1]_q+\frac{q^{a_1}}{[a_2]_{q^{-1}}+\dfrac{q^{-a_2}}z}.\]

  Taking $z=\infty$ gives the property for $p=1$. The general case follows by induction on~$p$. Indeed, if the result is true for $[a_3,\dots,a_{2p}]$, then it is also true for~$[a_1,\dots,a_{2p}]$, since
  \[[a_1,\dots,a_{2p}]=A_{a_1}A_{a_2}[a_3,\dots,a_{2p}],\
    \text{which implies}\ [a_1,\dots,a_{2p}]^\sharp_q=[A_{a_1}A_{a_2}]_q\cdot[a_3,\dots,a_{2p}]^\sharp_q.\qedhere\]
\end{proof}

\section{Weighted paths on snake graphs}

Let $a_1$, $a_2$ be positive integers. The \emph{hook} associated to~$(a_1,a_2)$ is the directed graph~$G(a_1,a_2)$ built by glueing~$a_1$ square boxes on a horizontal line, then a pile of~$a_2-1$ boxes above the rightmost box. The edges of the graph are the sides of the boxes, and they are directed toward the right or the top (figure~\ref{fig:snake}, left).

Let $(a_1,\dots,a_{2p})$ be an even-length sequence of positive integers. The corresponding snake graph, denoted by $G(a_1,\dots,a_{2p})$, is obtained by starting with the hook associated to~$(a_1,a_2)$, stacking a box on top of it, then glueing the hook associated to~$(a_3,a_4)$, and so forth (figure~\ref{fig:snake}, right). This graph is implicitly embedded in the plane, so that the coordinate of the bottom left corner is at the origin.

\newcommand{\snakeO}[2]{%
  \draw[>=latex,->] (#1,#2)--++(1,0);
  \draw[>=latex,->] (#1+1,#2)--++(0,1);
  \draw[>=latex,->] (#1,#2)--++(0,1);
  \draw[>=latex,->] (#1,#2+1)--++(1,0);
}
\newcommand{\snakeD}[2]{%
  \draw[>=latex,->] (#1,#2)--++(1,0);
  \draw[>=latex,->] (#1+1,#2)--++(0,1);
  \draw[>=latex,->] (#1,#2+1)--++(1,0);
}
\newcommand{\snakeN}[2]{%
  \draw[>=latex,->] (#1+1,#2)--++(0,1);
  \draw[>=latex,->] (#1,#2)--++(0,1);
  \draw[>=latex,->] (#1,#2+1)--++(1,0);
}
\newcommand{\snakeGa}[2]{%
  \fill[gray!20] (#1,#2) rectangle (#1+1,#2+1);
  \draw[>=latex,->] (#1,#2)--++(0,1);
  \draw[>=latex,->] (#1,#2+1)--++(1,0);
  \draw[>=latex,->] (#1,#2)--++(1,0);
}
\newcommand{\snakeGan}[2]{%
  \draw[>=latex,->] (#1,#2)--++(0,1);
  \draw[>=latex,->] (#1,#2+1)--++(1,0);
  \draw[>=latex,->] (#1,#2)--++(1,0);
}
\newcommand{\snakeHook}[4]{%
  \snakeO{#1}{#2}
  \foreach \d in {2,...,#3}{
    \snakeD{#1+\d-1}{#2}
  }
  \foreach \d in {2,...,#4}{
    \snakeN{#1+#3-1}{#2+\d-1}
  }
}

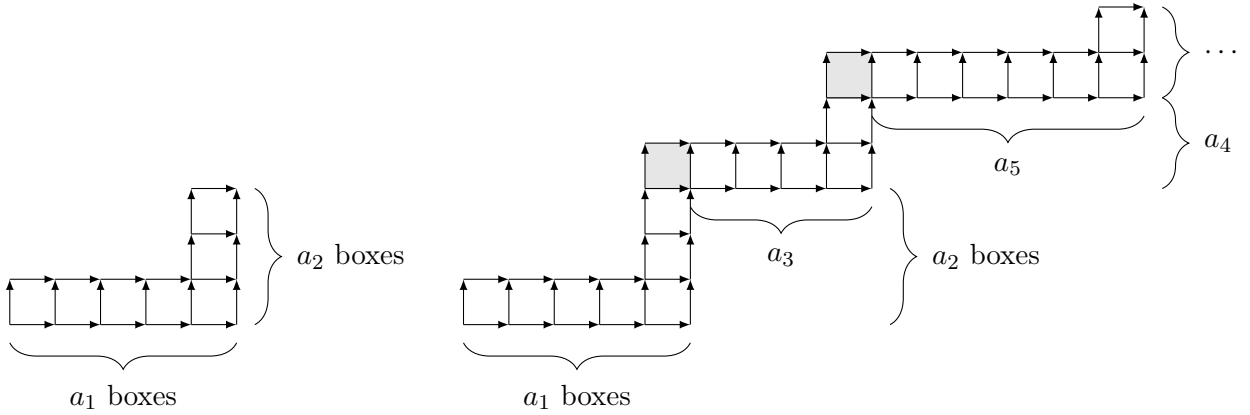
\begin{figure}[ht]
  \centering
  \begin{tikzpicture}[x=.6cm, y=.6cm]
    \snakeHook{0}{0}{5}{3}
      \draw[decorate, decoration={brace, mirror, amplitude=10pt}] (0,-0.4) -- (5,-0.4) node[midway, below=12pt] {$a_1$ boxes};
      \draw[decorate, decoration={brace, mirror, amplitude=10pt}] (5.4,0) -- (5.4,3) node[midway, right=12pt] {$a_2$ boxes};
    \begin{scope}[xshift=6cm]
      \snakeHook{0}{0}{5}{3}
      \snakeGa{4}{3}
      \snakeHook{5}{3}{4}{2}
      \snakeGa{8}{5}
      \snakeHook{9}{5}{6}{2}
      \draw[decorate, decoration={brace, mirror, amplitude=10pt}] (0,-0.4) -- (5,-0.4) node[midway, below=12pt] {$a_1$ boxes};
      \draw[decorate, decoration={brace, mirror, amplitude=10pt}] (5,3-0.4) -- (9,3-0.4) node[midway, below=12pt] {$a_3$};
      \draw[decorate, decoration={brace, mirror, amplitude=10pt}] (9,5-0.4) -- (15,5-0.4) node[midway, below=12pt] {$a_5$};
      \draw[decorate, decoration={brace, mirror, amplitude=10pt}] (9.4,0) -- (9.4,3) node[midway, right=12pt] {$a_2$ boxes};
      \draw[decorate, decoration={brace, mirror, amplitude=10pt}] (15.4,3) -- (15.4,5) node[midway, right=12pt] {$a_4$};
      \draw[decorate, decoration={brace, mirror, amplitude=10pt}] (15.4,5) -- (15.4,7) node[midway, right=12pt] {$\cdots$};
    \end{scope}
  \end{tikzpicture}
  \caption{The hook associated to $(5,3)$ and the snake graph associated to $(5,3,4,2,6,2)$}
  \label{fig:snake}
\end{figure}

Given two vertices~$\sg$ and~$\fg$, a \emph{path} from~$\sg$ (``start'') to~$\fg$ (``finish'') is a finite sequence of arrows, say $\ga=(\ga_1,\dots,\ga_l)$, such that the source of the first arrow~$\ga_1$ is~$\sg$, the target of~$\ga_l$ is~$\fg$, and the target of~$\ga_i$ is the source of~$\ga_{i+1}$.
The \emph{weight} of a path~$\ga$ is~$q^{\a(\ga)}$, where~$\a(\ga)$ is the area above or on the left of the path in the snake graph (a square counts for~$1$; see figure~\ref{fig:weight}). We denote by~$\Ga(\sg,\fg)$ the weighted sum of all paths from~$\sg$ to~$\fg$:
\[\Ga(\sg,\fg)=\sum_{\ga:\sg\to\fg}q^{\a(\ga)}.\]

\begin{figure}[ht!]
  \centering
  \begin{tikzpicture}[x=.6cm, y=.6cm]
    \fill[slateblue!20] (0,0) rectangle (3,1);
    \fill[slateblue!20] (4,1) rectangle (5,4);
    \fill[slateblue!20] (5,3) rectangle (8,4);
    \fill[slateblue!20] (8,5) rectangle (12,6);
    \fill[slateblue!20] (14,6) rectangle (15,7);
    \snakeHook{0}{0}{5}{3}
    \snakeGan{4}{3}
    \snakeHook{5}{3}{4}{2}
    \snakeGan{8}{5}
    \snakeHook{9}{5}{6}{2}
    \draw[very thick] (0,0)--++(3,0)--++(0,1)--++(2,0)--++(0,2)--++(3,0)--++(0,2)--++(4,0)--++(0,1)--++(3,0)--++(0,1);
    \fill (0,0) circle(2pt) node[left]{$\sg_1$};
    \fill (15,7) circle(2pt) node[right]{$\fg_1$};
    \fill (4,1) circle(2pt) node[above left]{$\sg_2$};
    \fill (14,7) circle(2pt) node[left]{$\fg_2$};
  \end{tikzpicture}
  \caption{A path of weight~$14$ from $\sg_1$ to $\fg_1$ on the snake graph associated to $(5,3,4,2,6,2)$}
  \label{fig:weight}
\end{figure}
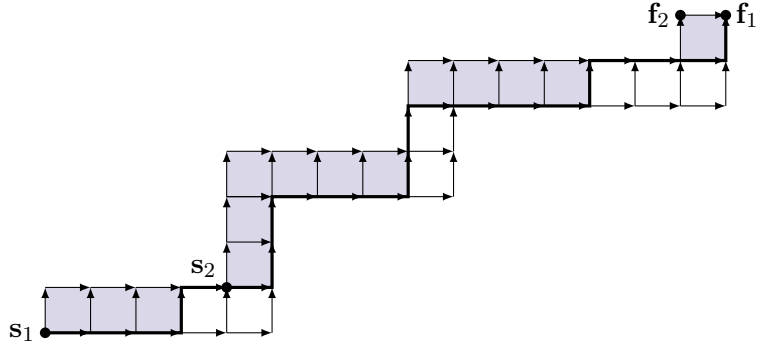

On the snake graph~$G(a_1,\dots,a_{2p})$, we define four special vertices: two starting points,~$\sg_1=(0,0)$ and~$\sg_2=(a_1-1,1)$, and two finishing points,~$\fg_1=(a_1+a_3+\cdots+a_{2p-1},a_2+a_4+\cdots+a_{2p})$ and~$\fg_2=(a_1+a_3+\cdots+a_{2p-1}-1,a_2+a_4+\cdots+a_{2p})$ (figure~\ref{fig:weight}).

\begin{prop}\label{A1A2q_chemins}
  Let $a_1$, $a_2$ be positive integers. Let~$\sg_1$, $\sg_2$, $\fg_1$, $\fg_2$ be the special vertices on the snake graph~$G(a_1,a_2)$. Up to a sign and a power of~$q$, one has
  \[[A_{a_1}A_{a_2}]_q=\begin{pmatrix}q[a_1]_q[a_2]_q+q^{a_1+a_2}&[a_1]_q\\q[a_2]_q&1\end{pmatrix}=\begin{pmatrix}q\Ga(\sg_1,\fg_1)&\Ga(\sg_1,\fg_2)\\q\Ga(\sg_2,\fg_1)&\Ga(\sg_2,\fg_2)\end{pmatrix}.\]
\end{prop}

\begin{proof}
  The first equality was proven in proposition~\ref{A1A2q}. We now compute the four weighted sums of paths on the right-hand side (see figure~\ref{fig:hook}).

  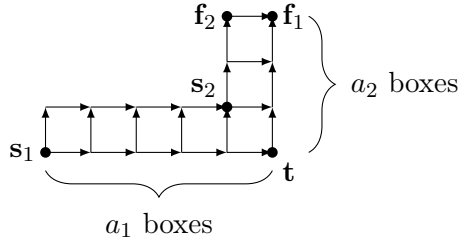
\begin{figure}[ht]
  \centering
  \begin{tikzpicture}[x=.6cm, y=.6cm]
    \snakeHook{0}{0}{5}{3}
      \draw[decorate, decoration={brace, mirror, amplitude=10pt}] (0,-0.4) -- (5,-0.4) node[midway, below=12pt] {$a_1$ boxes};
      \draw[decorate, decoration={brace, mirror, amplitude=10pt}] (5.8,0) -- (5.8,3) node[midway, right=12pt] {$a_2$ boxes};
    \fill (0,0) circle(2pt) node[left]{$\sg_1$};
    \fill (5,3) circle(2pt) node[right]{$\fg_1$};
    \fill (4,1) circle(2pt) node[above left]{$\sg_2$};
    \fill (5,0) circle(2pt) node[below right]{$\tg$};
    \fill (4,3) circle(2pt) node[left]{$\fg_2$};
    \end{tikzpicture}
    \caption{Starting and finishing points in a hook}
    \label{fig:hook}
  \end{figure}

  There are~$a_1$ paths from~$\sg_1$ to~$\sg_2$, parametrized by the abscissa~$k$ of the vertical step. Hence, one has $\Ga(\sg_1,\sg_2)=\sum_{k=0}^{a_{1}-1}q^k=[a_1]_q$. Similarly, $\Ga(\sg_2,\fg_1)=[a_2]_q$. There is a unique path from~$\sg_2$ to~$\fg_2$, the area it delimits is zero, so $\Ga(\sg_2,\fg_2)=1$. A path from~$\sg_1$ to~$\fg_2$ necessarily passes through~$\sg_2$, so $\Ga(\sg_1,\fg_2)=\Ga(\sg_1,\sg_2)\Ga(\sg_2,\fg_2)=[a_1]_q$.
  
  Any path from~$\sg_1$ to~$\fg_1$ must cross the line of points $(x,y)$ such that $x+y=a_1$ somewhere. Since there are only two vertices of~$G(a_1,a_{2})$ on this line, such a path passes either through~$\tg=(a_1,0)$, or through~$\sg_2=(a_1-1,1)$. The unique path $\sg_1\to\tg\to\fg_1$ has weight~$a_1+a_2-1$ (the total number of boxes). Let~$\ga$ be the concatenation of a path~$\ga_1$ from~$\sg_1$ to~$\sg_2$ and a path~$\ga_2$ from~$\sg_2$ to~$\fg_1$: it delimits an area $\a(\ga)=\a(\ga_1)+\a(\ga_2)$, so the weighted sum of paths $\ga:\sg_1\to\sg_2\to\fg_1$ is the product~$\Ga(\sg_1,\sg_2)\Ga(\sg_2,\fg_1)$. This gives the last coefficient:
  \[\Ga(\sg_1,\fg_1)=\Ga(\sg_1,\tg)\Ga(\tg,\fg_1)+\Ga(\sg_1,\sg_2)\Ga(\sg_2,\fg_1)=q^{a_1+a_2-1}+[a_1]_q[a_2]_q.\qedhere\]
\end{proof}

\begin{prop}\label{A1A2p}
    Let $(a_1,\dots,a_{2p})$ be an even-length sequence of positive integers. Let~$\sg_1$, $\sg_2$, $\fg_1$, $\fg_2$ be the special vertices on the snake graph~$G(a_1,\dots,a_{2p})$. Up to a power of~$q$, one has
    \[[A_{a_1}\cdots A_{a_{2p}}]_q=\begin{pmatrix}q\Ga(\sg_1,\fg_1)&\Ga(\sg_1,\fg_2)\\q\Ga(\sg_2,\fg_1)&\Ga(\sg_2,\fg_2)\end{pmatrix}.\]
\end{prop}
\begin{proof}
  We proceed by induction on~$p$. For~$p=1$, this is proposition~\ref{A1A2q_chemins}. Let $p\ge2$. We assume the assertion is true for sequences of length~$2(p-1)$. Let~$(a_1,\dots,a_{2p})$ be a sequence of positive integers. The snake graph~$G(a_1,\dots,a_{2p})$ is obtained from~$G(a_1,a_2)$ by stacking a box above, then glueing the snake graph~$G(a_3,\dots,a_{2p})$ on the right (figure~\ref{fig:snake_empile}).
  \begin{figure}[ht!]
    \centering
    \begin{tikzpicture}[x=.6cm, y=.6cm]
      \snakeHook{0}{0}{5}{3}
      \snakeGa{4}{3}
      \foreach \jj in {1,...,4} {
        \snakeD{4+\jj}{3}
      }
      \snakeN{8}{4}
      \draw[>=latex, ->] (5,3) --++(1,0);
      \draw[>=latex, ->] (5,3) --++(0,1);
      \draw[>=latex, ->] (5,4) --++(1,0);
      \draw[>=latex, ->, dashed] (8,5) --++(0,.5) arc(180:90:2) --++(2,0) arc(-90:0:1.5);
      \draw[>=latex, ->, dashed] (9,5) --++(0,.5) arc(180:90:1) --++(2,0) arc(-90:0:2.5);
      \draw[>=latex, ->] (13.5,9.)--++(1,0);
      \draw[>=latex, ->] (13.5,9)--++(0,1);
      \draw[>=latex, ->] (13.5,10) node[left]{$\fg_2$}--++(1,0);
      \draw[>=latex, ->] (14.5,9)--++(0,1) node[right]{$\fg_1$};
      \node[right] at (10,3.5) {$G(a_3,\dots,a_{2p})$};
      \node[right] at (6,1) {$G(a_1,a_{2})$};
    \fill (0,0) circle(2pt) node[left]{$\sg_1$};
    \fill (5,3) circle(2pt) node[below right]{$\fg'_1=\sg'_1$};
    \fill (8,4) circle(2pt) node[above left]{$\sg'_2$};
    \fill (4,1) circle(2pt) node[above left]{$\sg_2$};
    \fill (4,3) circle(2pt) node[left]{$\fg'_2$};
    \fill (4,4) circle(2pt) node[above left]{$\ug$};
    \end{tikzpicture}
    \caption{Stacking of a hook and a snake graph}
    \label{fig:snake_empile}
  \end{figure}
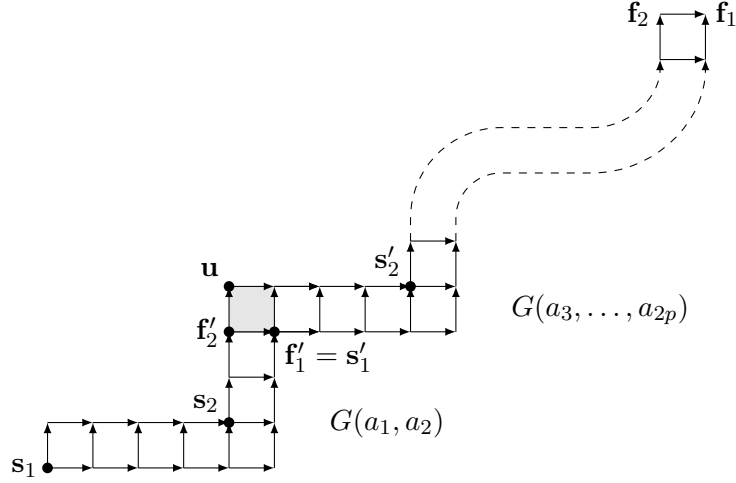

  Let $\fg'_1=(a_1,a_2)$ and $\fg'_2=(a_1-1,a_2)$ be the finishing points of the first hook~$G(a_1,a_2)$, and let~$\ug=(a_1-1,a_2+1)$ be the top left corner of the stacked box (figure~\ref{fig:snake_empile}). Let~$\sg'_2=(a_1+a_3-1,a_2+1)$: this is the second starting point of~$G(a_3,\dots,a_{2p})$ (after a translation); notice that the first one is $\sg'_1=\fg'_1$. With this notation, the induction hypothesis is that up to a sign and a power of~$q$,
  \[[A_{a_3}\cdots A_{a_{2p}}]_q=\begin{pmatrix}q\Ga(\sg'_1,\fg_1)&\Ga(\sg'_1,\fg_2)\\q\Ga(\sg'_2,\fg_1)&\Ga(\sg'_2,\fg_2)\end{pmatrix}.\]

  Let $i,j\in\{1,2\}$. Any path from~$\sg_i$ to~$\fg_j$ either must cross the line of points $(x,y)$ such that $x+y=a_1+a_2$ somewhere. Since there are only two vertices of~$G(a_1,\dots,a_{2p})$ on this line, such a path passes either through~$\fg'_1$ or through~$\ug$. On the one hand, any path through~$\ug$ passes through~$\fg'_2$ and~$\sg'_1$, and the section between these two points does not change the delimited area. Hence, the weighted sum of paths passing through~$\ug$ is $\Ga(\sg_i,\fg'_2)\Ga(\sg'_2,\fg_j)$. On the other hand, the weighted sum of a path through~$\fg'_1$ is $q\Ga(\sg_i,\fg'_1)\Ga(\sg'_1,\fg_j)$, where the extra~$q$ comes from the greyed box (recall $\sg'_1=\fg'_1$). Hence,
  \[\Ga(\sg_i,\fg_j)=q\Ga(\sg_i,\fg'_1)\Ga(\sg'_1,\fg_j)+\Ga(\sg_i,\fg'_2)\Ga(\sg'_2,\fg_j).\]
  This justifies the last equality of the following computation: up to a sign and a power of~$q$,
  \begin{align*}
    [A_{a_1}\cdots A_{a_{2p}}]_q
    &=[A_{a_1}A_{a_2}]_q[A_{a_3}\cdots A_{a_{2p}}]_q\\
    &=\begin{pmatrix}q\Ga(\sg_1,\fg'_1)&\Ga(\sg_1,\fg'_2)\\q\Ga(\sg_2,\fg'_1)&\Ga(\sg_2,\fg'_2)\end{pmatrix}
      \begin{pmatrix}q\Ga(\sg'_1,\fg_1)&\Ga(\sg'_1,\fg_2)\\q\Ga(\sg'_2,\fg_1)&\Ga(\sg'_2,\fg_2)\end{pmatrix}\\
    &=\begin{pmatrix}q^2\Ga(\sg_1,\fg'_1)\Ga(\sg'_1,\fg_1)+q\Ga(\sg_1,\fg'_2)\Ga(\sg'_2,\fg_1)&q\Ga(\sg_1,\fg'_1)\Ga(\sg'_1,\fg_2)+\Ga(\sg_1,\fg'_2)\Ga(\sg'_2,\fg_2)\\q^2\Ga(\sg_2,\fg'_1)\Ga(\sg'_1,\fg_1)+q\Ga(\sg_2,\fg'_2)\Ga(\sg'_2,\fg_1)&q\Ga(\sg_2,\fg'_1)\Ga(\sg'_1,\fg_2)+\Ga(\sg_2,\fg'_2)\Ga(\sg'_2,\fg_2)\end{pmatrix}\\
    &=\begin{pmatrix}q\Ga(\sg_1,\fg_1)&\Ga(\sg_1,\fg_2)\\q\Ga(\sg_2,\fg_1)&\Ga(\sg_2,\fg_2)\end{pmatrix}.
  \end{align*}
  This allows to conclude the recurrence.
\end{proof}

\begin{coro}
  If $m/n=[a_1,\dots,a_{2p}]$, the numerator of $[m/n]_q^\sharp$ is $\Ga(\sg_1,\fg_1)$ and the denominator of~$[m/n]_q^\sharp$ is~$\Ga(\sg_2,\fg_1)$.
\end{coro}

\begin{proof}
  By proposition~\ref{A1A2p},
  \[\left[\frac{m}{n}\right]_q^\sharp=[A_{a_1}\cdots A_{a_{2p}}]_q\cdot\infty
    =\frac{q\Ga(\sg_1,\fg_1)}{q\Ga(\sg_2,\fg_1)}=\frac{\Ga(\sg_1,\fg_1)}{\Ga(\sg_2,\fg_1)}.\]
  It is thus enough to show that these two polynomials are relatively prime. Since the determinant of the matrix in proposition~\ref{A1A2p} is of the form~$\eps q^k$ for some~$\eps\in\{-1,1\}$ and some~$k\in\ZM$, the gcd of~$\Ga(\sg_1,\fg_1)$ and~$\Ga(\sg_2,\fg_1)$ is of this form. But both polynomials evaluate to~$1$ at $q=0$, since there is a unique path of area zero from~$\sg_1$ or~$\sg_2$ to~$\fg_1$ (follow the top border of the snake graph). This implies the claim: the gcd is~$\pm1$ and~$\Ga(\sg_1,\fg_1)$ and~$\Ga(\sg_2,\fg_1)$ are indeed the numerator and the denominator of~$[m/n]_q^\sharp$.
\end{proof}

\section{Modified weighted paths on snake graphs}

To interpret the numerator and the denominator of the left $q$-deformed rationals~$[m/n]_q^\flat$, one must slightly change the weight function on the snake graphs. The modification consists in attributing an area equal to~$2$ to the top right box of the graph. Visually, this amounts to transforming the corresponding square into a rectangle. For instance, figure~\ref{fig:mod-snake-2} pictures~$G^\flat(5,1)$, $G^\flat(5,2)$ and~$G^\flat(5,4)$.

\begin{defi}
  Let~$(a_1,\dots,a_{2p})$ be an even-length sequence of positive integers. The modified snake graph~$G^\flat(a_1,\dots,a_{2p})$ is the graph~$G(a_1,\dots,a_{2p})$ equipped with a modified weight function~$\a^\flat$: the weight of a path~$\ga$ from a vertex~$\sg$ to a vertex~$\fg$ is
  \[\a^\flat(\ga)=
    \begin{cases}
      \a(\ga)+1&\text{if $\fg=\fg_1$ and the box containing~$\fg_1$ is above~$\ga$};\\
      \a(\ga)&\text{otherwise.}
    \end{cases}\]

  Given~$G^\flat(a_1,\dots,a_{2p})$ and two vertices~$\sg$ and~$\fg$, the modified weighted sum is
  \[\Ga^\flat(\sg,\fg)=\sum_{\ga:\sg\to\fg}q^{\a^\flat(\ga)}.\]
    
  As in $G(a_1,\dots,a_{2p})$, there are four special vertices~$\sg_1$, $\sg_2$, $\fg_1$, $\fg_2$ (figure~\ref{fig:mod-snake-2}).
\end{defi}

    \begin{figure}[ht!]
    \centering
    \begin{tikzpicture}[x=.6cm, y=.6cm]
      \draw[>=latex, ->] (0,0) --++(0,1);
      \foreach \k in {0,1,2,3} {
        \snakeD{\k}{0}
      }
      \draw[>=latex, ->] (4,0) --++(2,0);
      \draw[>=latex, ->] (4,1) --++(2,0);
      \draw[>=latex, ->] (6,0) --++(0,1);
      \draw[decorate, decoration={brace, mirror, amplitude=10pt}] (0,-0.4) -- (6,-0.4) node[midway, below=12pt] {$a_{2p-1}$ boxes};
      \node at (3,-2.5) {$G^\flat(a_1,\dots,a_{2p-1},1)$};
    \draw[>=latex, ->] (9,0) --++(0,1);
      \foreach \k in {0,1,2,3,4} {
        \snakeD{9+\k}{0}
      }
    \draw[>=latex, ->] (13,1)--++(0,2);
    \draw[>=latex, ->] (14,1)--++(0,2);
    \draw[>=latex, ->] (13,3)--++(1,0);
     \fill (0,0) circle(2pt) node[left]{$\sg_1$};
    \fill (6,1) circle(2pt) node[right]{$\fg_1$};
    \fill (4,1) circle(2pt) node[above]{$\sg_2=\fg_2$};
    \fill (9,0) circle(2pt) node[left]{$\sg_1$};
    \fill (14,3) circle(2pt) node[right]{$\fg_1$};
    \fill (13,1) circle(2pt) node[above left]{$\sg_2$};
    \fill (13,3) circle(2pt) node[left]{$\fg_2$};
      \draw[decorate, decoration={brace, mirror, amplitude=10pt}] (9,-0.4) -- (14,-.4) node[midway, below=12pt] {$a_{2p-1}$ boxes};
      \node at (11.5,-2.5) {$G^\flat(a_1,\dots,a_{2p-1},2)$};
    \snakeHook{17}{0}{5}{3}
    \draw[>=latex, ->] (21,3)--++(0,2);
    \draw[>=latex, ->] (22,3)--++(0,2);
    \draw[>=latex, ->] (21,5)--++(1,0);
    \fill (17,0) circle(2pt) node[left]{$\sg_1$};
    \fill (22,5) circle(2pt) node[right]{$\fg_1$};
    \fill (21,1) circle(2pt) node[above left]{$\sg_2$};
    \fill (21,5) circle(2pt) node[left]{$\fg_2$};
      \draw[decorate, decoration={brace, mirror, amplitude=10pt}] (17,-0.4) -- (22,-.4) node[midway, below=12pt] {$a_{2p-1}$ boxes};
      \draw[decorate, decoration={brace, mirror, amplitude=10pt}] (22.9,0) -- (22.9,5) node[midway, right=12pt] {$a_{2p}$ boxes};
      \node at (19.5,-2.5) {$G^\flat(a_1,\dots,a_{2p-1},a_{2p})$, $a_{2p}>2$};
    \end{tikzpicture}
    \caption{Modified snake graphs (only the last hook is represented)}
    \label{fig:mod-snake-2}
  \end{figure}
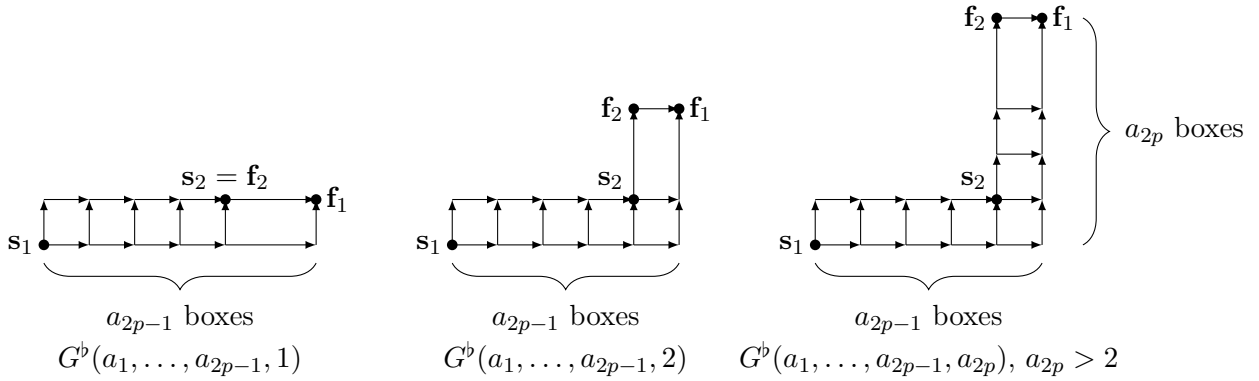


The reason for increasing the area of the top right box is the following lemma. Recall that for a positive integer~$a$,
\[ [a]_q^\flat=q^{a}+q^{a-2}+\cdots+q+1\quad\text{and}\quad [a]_q=[a]_q^\sharp=q^{a-1}+q^{a-2}+\cdots+q+1.\]

\begin{lemm}
  Let $a_1$, $a_2$ be positive integers. 
  \begin{enumerate}[{\em(i)}]
  \item The modified weighted sums in $G^\flat(a_1,a_2)$ are:
    \[\begin{pmatrix}\Ga^\flat(\sg_1,\fg_1)&\Ga^\flat(\sg_1,\fg_2)\\\Ga^\flat(\sg_2,\fg_1)&\Ga^\flat(\sg_2,\fg_2)\end{pmatrix}=
    \begin{pmatrix}q^{a_1+a_2}+[a_1]_q\bigl(q[a_2]_q+1-q\bigr)&[a_1]_q\\q[a_2]_q+1-q&1\end{pmatrix}.\]
  \item Up to a sign and a power of~$q$, one has the following matrix product:
    \[ [A_{a_1}A_{a_2}]_q\begin{pmatrix}1\\1-q\end{pmatrix}
      =\begin{pmatrix}\Ga^\flat(\sg_1,\fg_1)\\\Ga^\flat(\sg_2,\fg_1)\end{pmatrix}.\]
  \end{enumerate}
\end{lemm}
\begin{proof}
  (i) The weighted sum of paths from~$\sg_1$ to~$\fg_2$ is the same as in the usual snake graph, that is, $\Ga^\flat(\sg_1,\fg_2)=[a_1]_q$. There is a unique path from~$\sg_2$ to~$\fg_2$: it is vertical and it delimits a null area, so that $\Ga^\flat(\sg_2,\fg_2)=1$.

  There are~$a_2$ paths from~$\sg_2$ to~$\fg_1$, parametrized by the ordinate of the horizontal step, which belongs to $\{1,2,\dots,a_{2}-1,a_2+1\}$. Taking into account that the highest box contributes for an area of~$2$, this implies that
  \[\Ga^\flat(\sg_2,\fg_1)=\begin{cases}q^{a_2}+q^{a_2-1}+\cdots+q^2+1&\text{if}\ a_2\ge2;\\1&\text{if}\ a_2=1.\end{cases}\]

  In both cases, one has $\Ga^\flat(\sg_2,\fg_1)=q[a_2]_q+1-q$.
  
  Eventually, there are two kinds of paths from~$\sg_1$ to~$\fg_1$: the ones that pass through~$\sg_2$, and the (unique) one that passes through the bottom right corner, which has weight~$q^{a_1+a_2}$. This shows that
  \[\Ga^\flat(\sg_1,\fg_1)=\Ga(\sg_1,\sg_2)\Ga(\sg_2,\fg_1)+q^{a_1+a_2}
    =[a_1]_q\bigl(q[a_2]_q+1-q\bigr)+q^{a_1+a_2}.\]

  (ii) On the matrix side, this is an easy computation by proposition~\ref{A1A2q}:
  \[ [A_{a_1}A_{a_2}]_q\begin{pmatrix}1\\1-q\end{pmatrix}
    =\begin{pmatrix}q[a_1]_q[a_2]_q+q^{a_1+a_2}+(1-q)[a_1]_q\\
       q[a_2]_q+1-q\end{pmatrix}
    =\begin{pmatrix}[a_1]_q\bigl(q[a_2]_q+1-q\bigr)+q^{a_1+a_2}\\
       q[a_2]_q+1-q\end{pmatrix},  \]
  the same as $\binom{\Ga^\flat(\sg_1,\fg_1)}{\Ga^\flat(\sg_2,\fg_1)}$, which was computed just above.
\end{proof}

\begin{coro}
  If $m/n=[a_1,\dots,a_{2p}]$, the numerator of $[m/n]_q^\flat$ is $\Ga^\flat(\sg_1,\fg_1)$ and the denominator of~$[m/n]_q^\flat$ is~$\Ga^\flat(\sg_2,\fg_1)$.
\end{coro}

\begin{proof}
  Use the same arguments as in proposition~\ref{A1A2p} with the following minor differences: replace the snake graph~$G(a_1,\dots,a_{2p})$ by~$G^\flat(a_1,\dots,a_{2p})$ and forget about~$\fg_2$.\footnote{To be completed.}
\end{proof}

\bibliographystyle{plain}
\bibliography{TIPE}

\end{document}